\documentclass[11pt,a4paper,twoside]{article}
\usepackage{enumerate}
\usepackage{amsmath}
\usepackage{fix-cm}
\usepackage{pb-diagram}
\usepackage[english]{babel}
\usepackage{amssymb,latexsym}
\usepackage{hyperref}
\usepackage{color}
\usepackage{authblk}
\DeclareMathOperator{\vol}{Vol}

\def \bui#1#2{\mathrel{\mathop{\kern 0pt#1}\limits^{#2}}}
\def \buil#1#2{\mathrel{\mathop{\kern 0pt#1}\limits_{#2}}}

\let\<\langle
\let\>\rangle
\newcommand{\R}{{\mathbb R}}

\newtheorem{example}{Examples}[section]
\newtheorem{thm}{Theorem}[section]
\newtheorem{lemma}[thm]{Lemma}
\newtheorem{prop}[thm]{Proposition}

\newtheorem{cor}[thm]{Corollary}
\newtheorem{remark}[thm]{Remark}
\newtheorem{remarks}[thm]{Remarks}
\newtheorem{definition}[thm]{Definition}
\newtheorem{notation}[thm]{Notation}
\newtheorem{exabout:ample}[thm]{Example}

\title{The Wentzell operator on differential forms}

\author{Fida El Chami}

\affil{\footnotesize Lebanese University, Faculty of Sciences II, Department of Mathematics, P.O. Box 90656 Fanar-Matn, Lebanon\\\vspace{3pt}
 \texttt{fchami@ul.edu.lb}}

\begin{document}
\maketitle
\noindent\begin{center}\begin{tabular}{p{115mm}}
\begin{small}{\bf Abstract.}
We extend the Wentzell operator to differential forms on compact manifolds with boundary, investigate its spectrum, and establish a variational characterization of the first eigenvalue. We explicitly compute the eigenvalues on the unit ball in Euclidean space. We also derive several eigenvalue estimates for the first eigenvalue, some of which are shown to be sharp. Finally, we obtain a lower bound for the gap between the lowest eigenvalues in consecutive degrees. 
\end{small}
\end{tabular}\end{center}

\noindent\begin{small}{\it Mathematics Subject Classification} (2020): 53C21, 58C40, 58J32, 58J50
\end{small}

\noindent\begin{small}{\it Keywords}: manifolds with boundary, Wentzell operator, boundary value problem, eigenvalue estimates   	
\end{small}

\section{Introduction}
Let $(M^{n+1},g)$ be a compact Riemannian manifold of dimension $n+1$ with nonempty smooth boundary $\partial M$. We denote by $\nu$ the inward unit normal vector field to the boundary. For $\beta \in \R$, the eigenvalue problem for Wentzell boundary conditions is defined as follows:
\begin{equation}
    \label{scalar_wentzell}
    \begin{cases}
        \Delta u=0 & \textrm{on } M,\\
        -\dfrac{\partial u}{\partial \nu} +\beta \, \Delta^{\partial M} u=\lambda u & \textrm{on } \partial M,
    \end{cases}
\end{equation}
where $\Delta$ and $\Delta^{\partial M}$ denote the Laplace-Beltrami operator on $M$ and $\partial M$, respectively. 
The coefficient $\beta$ appearing in the Wentzell boundary condition can be interpreted as a tangential diffusion parameter, arising as the limit of a thin, highly conductive coating layer as its thickness tends to zero 
\cite{BL1999}. 
Wentzell boundary conditions describe a coupling between the dynamics in the interior of a domain and the phenomena occurring on its boundary. In fact, the eigenvalue problem \eqref{scalar_wentzell} can be reformulated as $\mathcal{L}_\beta u=\lambda u$ where $\mathcal{L}_\beta$ is the Wentzell operator, given by the sum of the Dirichlet-to-Neumann (or Steklov) operator and of $\beta \Delta^{\partial M}$. In the case where $\beta=0$, the operator $\mathcal{L}_\beta$ reduces to the Steklov operator. 

Originally introduced by Feller \cite{Feller1952,Feller1954} in the study of one-dimensional diffusion processes and subsequently extended by Venttsel \cite{Wentzell1959} (now commonly spelled Wentzell) to the multidimensional setting, Wentzell-type boundary conditions have since become a classical topic, studied from several perspectives, including semigroup theory, partial differential equations, and spectral geometry. 
These questions 
have in particular been extensively investigated over the past decades, as the coupling between the interior operator and the boundary conditions naturally raises questions about how the spectrum encodes the geometry of the underlying domain. 
In this direction, Dambrine, Kateb, and Lamboley \cite{DKL2016} studied the first nontrivial eigenvalue of the Wentzell-Laplace operator and its dependence on the geometry of the domain. 
They also computed the spectrum on the Euclidean ball, allowing them to characterize the equality cases of their estimates. Further contributions 
established eigenvalue estimates and rigidity results for Wentzell-type problems under various geometric hypotheses on compact manifolds with boundary 
\cite{DuWangXia2018,XiaWang2018,DuMaoWangXia2018, Ndiaye2022, DengDuMaoZhao2026}. 

It is natural to ask whether the Wentzell operator can be extended to differential forms, and which  results from the scalar case have analogues in this setting. The Dirichlet-to-Neumann operator acting on differential forms (see~\cite[Subsection~1.1] {RS_DtN2012}, and the references therein) plays a key role in defining and studying the Wentzell operator $\mathcal{L}^{[p]}_\beta$ on $p$-forms (see Section~\ref{Sec:Wentzell_on_forms}), in analogy with the scalar case. 
Although the spectral analysis of the operator $\mathcal{L}^{[p]}_\beta$ follows the classical variational framework, a first difficulty lies in the choice of the appropriate functional space and the bilinear form. Indeed, 
it turns out that the space $H$ naturally associated with the Wentzell operator, consistently with the scalar case, is 
\begin{equation}
\label{H}    
H=\left\{ \omega \in H^1(M) \mid \iota^* \omega \in H^1(\partial M) \textrm{ and } \nu \lrcorner \, \omega=0\right\}.\end{equation}
However, 
the bilinear form associated with $\mathcal{L}^{[p]}_\beta$ is not coercive on $H$. We therefore first work on a suitable space $V$, where the bilinear form can be modified to ensure coercivity (\eqref{a_beta} and Theorem~\ref{thm:coerc_abeta}). We then 
prove that the variational problem on $H$ is equivalent to the problem formulated on $V$. We next derive a variational characterization for the first eigenvalue of the Wentzell problem, along with explicit computations and several estimates. Some of our results recover 
the corresponding results for $p=0$; see, for example, Proposition \ref{min_Wentz_DtN_Delta_bound}. Note that the eigenvalue problem associated with the Wentzell operator is not invariant under the Hodge star operator. This leads us to introduce a dual problem (Definition \ref{def:dual_went}), which is useful, in particular, to obtain upper bounds for the first eigenvalue for harmonic domains (Propositions~\ref{upp_bound_n_harmonic} and~\ref{upp_bound_arb_degree_harmonic}).

Our study of the eigenvalues of the Wentzell operator builds on corresponding results for the Dirichlet-to-Neumann operator established in \cite{RS_DtN2012} and \cite{RS_DtN2014}. For instance, under suitable curvature assumptions, Raulot and Savo~ \cite{RS_DtN2012} established the lower bounds 
$$\begin{cases}\medskip
    \nu_{1,p} (M) > \frac{n-p+2}{n-p+1}\sigma_p(\partial M), &  1 \le p < \frac{n+1}{2}, \\\medskip
    \nu_{1,p} (M) \ge \frac{p+1}{p}\sigma_p(\partial M), &  p \ge \frac{n+1}{2},
\end{cases}$$
        for the eigenvalues of the Dirichlet-to-Neumann operator, with a characterization of the equality case in the second inequality. Moreover, assuming that the Ricci curvature of $M$ is nonnegative and that the boundary is mean-convex, they obtained the inequality $\nu_{1,n}(M) \ge (n+1)H_0$, where $H_0 \ge 0$ is a lower bound of the mean curvature,  and proved that equality holds if and only if $M$ is a Euclidean ball.      
We generalize these results to the eigenvalues of the Wentzell operator in Propositions~\ref{lower_bounds_wentzell} and \ref{lower_bound_Went_n}. 
For harmonic domains, the inequality $$\nu_{1,n}(M) \le \frac{\mathrm{Vol}(\partial M)}{\mathrm{Vol}(M)},$$
 obtained in \cite{RS_DtN2012}, is also generalized to the first Wentzell eigenvalue in degree $n$ (see Proposition~\ref{upp_bound_n_harmonic}).

The paper is organized as follows. In Section~\ref{Sec:Wentzell_on_forms}, we extend the Wentzell operator  to differential forms, and investigate its spectral properties. We establish the inequalities and compactness results needed to make the paper self-contained and provide a variational characterization of its first eigenvalue. In Section~\ref{Sec:Wentz_spectrum_Euclidean_ball}, we explicitly compute the eigenvalues of the Wentzell operator $\mathcal{L}^{[p]}_\beta$ on the Euclidean unit ball. Section~\ref{Sec:Wentz_eigen_estimates} is devoted to estimates of the first eigenvalue $\lambda_{1,p,\beta}(M)$ of $\mathcal{L}^{[p]}_\beta$. In particular, we establish a lower bound of $\lambda_{1,p,\beta}(M)$ in terms of the first eigenvalue of the Dirichlet-to-Neumann operator and of that of the Hodge Laplacian of the boundary. We also derive a lower bound involving geometric quantities depending on the second fundamental form of the boundary. 
Furthermore, for harmonic domains, we obtain upper bounds for the first eigenvalue 
expressed in terms of the isoperimetric ratio $\dfrac{\textrm{Vol}(\partial M)}{\textrm{Vol}( M)}$, with additional geometric quantities in some cases. Finally, in   Section~\ref{Sec:Wentz_gap_estimate}, we derive a lower bound for the gap between the eigenvalues of the Wentzell operator in degrees $p$ and $p-1$.  
\textbf{Acknowledgments}: I would like to thank N.~Ginoux, G.~Habib, O.~Makhoul and S.~Raulot for the helpful discussions related to this work. 
\section{The Wentzell Laplacian on differential forms}
\label{Sec:Wentzell_on_forms}
Let $(M^{n+1},g)$ be a compact Riemannian manifold of dimension $n+1$ with smooth boundary $\partial M$ and let $\nu$ be the inward unit normal vector field to the boundary. For every $p$-form $\phi$ on $\partial M$, we denote by $\hat{\phi}$ the unique harmonic tangential extension of $\phi$ on $M$ (see \cite{Sc} for the existence and uniqueness of such an extension). The form $\hat{\phi}$ is the unique solution to the boundary value problem 
\begin{equation}\label{eq:bvpharmonictangext}
\left\{\begin{array}{lll}\Delta \hat{\phi}&=0&\textrm{ on
}M\\ \nu \lrcorner \, \hat{\phi}&=0&\textrm{ on
}\partial M\\
\iota^* \hat{\phi}&=\phi &\textrm{ on
}\partial M.
\end{array}
\right.
\end{equation}
Recall that the Dirichlet-to-Neumann operator acting on $p$-forms on $\partial M$ is defined as follows \cite{RS_DtN2012}: for every $\phi \in \Omega^{p}(\partial M)$,  
\begin{equation}
\label{Dirichlet-to-Neumann}
\mathcal{T}^{[p]}(\phi)=- \nu \lrcorner \, d \hat{\phi},
\end{equation}
where $\hat{\phi}\in \Omega^p(M)$ is the solution to~\eqref{eq:bvpharmonictangext}. 

Let $\beta \in \R$. For $p=0, \dots, n$, we define the Wentzell-Laplace operator, also called the Wentzell Laplacian or the Wentzell operator, $\mathcal{L}^{[p]}_\beta$ on $\Omega^p(\partial M)$ by 
\begin{equation}
\label{wentzell_operator}
   \mathcal{L}^{[p]}_\beta= \mathcal{T}^{[p]} + \beta \,\Delta^{\partial M} . 
\end{equation}
Here $\Delta^{\partial M}$ denotes the Hodge Laplacian acting on differential forms on the boundary. Note that for $\beta=0$, the operator $\mathcal{L}^{[p]}_\beta $ reduces to the Dirichlet-to-Neumann operator acting on differential forms, and for $p=0$, it reduces to the scalar Wentzell Laplacian $\mathcal{L}_\beta$ recalled in the introduction. 
In what follows, we assume $\beta>0$.  

In order to study the spectral properties of the Wentzell operator \eqref{wentzell_operator}, let us introduce some notation and recall a few useful results.  
Let ${H}^k \Omega^p(M)$ be the space of $p$-forms on $M$ of Sobolev class $H^k$. It will be denoted by $H^k(M)$ instead of ${H}^k \Omega^p(M)$ when the context is clear. 
The absolute cohomology space is isomorphic to $\mathcal{H}^p_{N}(M)=\{\omega \in \Omega^p (M) \,| \, d\omega=\delta \omega =0, \nu \lrcorner \,\omega=0\}$. 
We also denote by $\mathcal{H}^p(\partial M)$ the space of harmonic $p$-forms on $\partial M$. Finally, for $\omega, \omega' \in H^1(M)$, the Dirichlet integral is defined by
\begin{equation}
    \label{Dirichlet_integral}
    \mathcal{D}(\omega,\omega')=\int_M (\<  d\omega,d\omega'\> + \<  \delta \omega,\delta\omega'\>) \, d \mu_g.
\end{equation}
Similarly, we denote by $\mathcal{D}^{\partial M}$ the Dirichlet integral on the boundary. 

Now, let 
\begin{equation}
  Z:=\{\omega \in \Omega^p(M) \mid \Delta \omega =0 \;\textrm{on}
\;M\textrm{ and } \nu \lrcorner \, \omega=0 \;\textrm{on}
\;\partial M\}.  
\end{equation}
We define the following inner product on $Z$: for all $\omega, \omega'\in Z$, 
\begin{equation}
\label{scalar_product}
    (\omega, \omega')_V= (\omega, \omega')_{H^1(M)} + (\iota^* \omega, \iota^* \omega')_{H^1(\partial M)},
\end{equation} 
and we denote by $V$ the completion of $Z$ with respect to the inner product $(.,.)_V$.
\begin{remark}
\label{properties_V}
Let $f \in L^2(M)$ be a $p$-form on $M$. Recall that $\omega \in H^1(M)$ is a weak solution to $\Delta \omega =f$ on $M$ if 
\begin{equation}\label{eq:defweaksol}
\mathcal{D}(\omega,\omega') =\left(f , \omega'\right)_ {
L^2(M)} \quad \forall \, \omega'\in H_0^1(M).
\end{equation}
We have the following inclusion:
 $$V \subset \left\{\omega \in H^1(M) \mid \iota^* \omega \in H^1(\partial M), \, \nu \lrcorner \, \omega =0 \textrm{
on } \partial M, \textrm{ and }
\Delta \omega =0 \textrm{ weakly on }M\right\}.$$
\end{remark}
A key tool is the partial integration formula that will be useful in the sequel: for all $\omega \in \Omega^p(M)$ and $\omega' \in  \Omega^{p+1}(M)$
\begin{equation*}
    \int_M \<  d \omega , \omega' \> \, d \mu_g =\int_M \<  \omega , \delta \omega' \> \, d \mu_g -\int_{\partial M} \<  \iota^* \omega, \nu \lrcorner \,\omega' \> \, d \mu_g.
\end{equation*} 
As a consequence, we obtain for all $\omega, \omega'\in \Omega^p(M)$
\begin{equation}
\label{partial_integration_laplacian}
    \int_M\< \Delta\omega,\omega'\> \, d\mu_g  =  \mathcal{D}(\omega,\omega')+\int_{\partial M}\Big(\<  \nu\lrcorner \, d\omega,\iota^*\omega'\>-\< \iota^*\delta\omega,\nu\lrcorner\,\omega'\>\Big) d\mu_g.
\end{equation}

We will also use Gaffney's inequality \cite[Cor. 2.1.6]{Sc}: there exists $C >0$ such that
\begin{equation}
    \label{ineq:Gaffney}
   \|\omega\|_{H^1(M)}^2 \le C \left(\mathcal{D}(\omega, \omega)+ \|\omega\|_{L^2(M)}^2 \right) \quad \forall \omega \in {H}^1 (M).
\end{equation}
Note that this result applies in particular to differential forms on the boundary $\partial M$, that is, there exists $c >0$ such that
\begin{equation}
    \label{ineq:Gaffney_partial_M}
    \|\phi\|_{H^1(\partial M)}^2 \le c \left( \mathcal{D}^{\partial M}(\phi, \phi) + \|\phi\|_{L^2(\partial M)}^2 \right)  \quad \forall \phi \in {H}^1(\partial M).
\end{equation}

\begin{thm}
\label{thm:coerc_abeta}
    Let $\beta$ be a positive real number. The bilinear form $a_\beta$ defined on the Hilbert space $V$ by:
    \begin{equation}
    \label{a_beta}
    a_\beta \, (\omega, \omega') = \mathcal{D}(\omega,\omega') + \beta \, \mathcal{D}^{\partial M}(\iota^* \omega,\iota^*\omega') + (\iota^* \omega,\iota^*\omega')_{L^2(\partial M)} \quad \forall \, \omega, \omega'\in V,
    \end{equation}
    is bounded and coercive on $V$. 
\end{thm}
{\it Proof.} The continuity of $a_\beta$ is obvious.  Moreover, the bilinear form $a_\beta$ is positive definite on $V$. Indeed, if $a_\beta(\omega,\omega)=0$ for some $\omega \in V$, then $d\omega =\delta\omega=0$ on $M$ and $\iota^*\omega=0$ on $\partial M$. Combined with the property $\nu \lrcorner \, \omega =0$, this implies that $\omega=0$. Now, applying  \eqref{ineq:Gaffney} and \eqref{ineq:Gaffney_partial_M} to $\omega \in V$ and $\iota^* \omega$ respectively, and then summing, we obtain:
\begin{eqnarray}
\label{ineq:Gaffney_a_beta}
    \|\omega\|_V^2 &\le & C\left(\mathcal{D}(\omega, \omega)+ \|\omega\|_{L^2(M)}^2 + \mathcal{D}^{\partial M}(\iota^* \omega,\iota^*\omega) + \|\iota^* \omega\|_{L^2(\partial M)}^2 \right) \nonumber\\
     &\le & C\left(a_\beta(\omega, \omega) + \|\omega\|_{L^2(M)}^2\right),
    \end{eqnarray}
for some constant $C>0$. Now, let
$$S=\{\omega \in V \mid \|\omega\|_{L^2(M)}=1\}.$$
The set $S$ is closed in $V$ by the continuity of the natural inclusion map $V \rightarrow L^2(M)$. Therefore, there exists a sequence $(\eta_i)_{i\in \mathbb{N}}$ in $S$ such that 
$a_\beta (\eta_i,\eta_i) $ converges to $\displaystyle\inf_{\omega \in S} a_\beta(\omega, \omega)$ as $i \to \infty$. By \eqref{ineq:Gaffney_a_beta}, the sequence $(\eta_i)_{i\in \mathbb{N}}$ is bounded in the Hilbert space $V$, hence, it admits a subsequence $(\eta_{i_k})_{k\in \mathbb{N}}$, weakly convergent to some $\eta$ in $V$. Note that the inclusion map $V \rightarrow L^2(M)$ is compact by the Rellich-Kondrachov theorem. Hence, there exists a subsequence $(\eta_{i_l})_{l\in \mathbb{N}}$ of $(\eta_{i_k})_{k\in \mathbb{N}}$
strongly convergent to some $\eta'$ in $L^2(M)$. Therefore, $\eta'$ must be equal to $\eta$, so $\|\eta\|_{L^2(M)}=1$, implying that $a_\beta(\eta, \eta) >0$. Consequently, 
$$a_\beta (\omega,\omega) \ge a_\beta (\eta,\eta) \, \|\omega\|_{L^2(M)}^2 \quad \forall \omega \in V.$$
Substituting the previous inequality into \eqref{ineq:Gaffney_a_beta}, we get for every $\omega \in V$
$$\|\omega\|_V^2 \le C \left( a_\beta(\omega , \omega) +\frac{a_\beta(\omega , \omega)}{a_\beta (\eta,\eta)}\right),$$
from which the coercivity of $a_\beta$ follows. 
\hfill$\square$

\vspace{0.3em}

For $\phi \in L^2(\partial M)$, the linear form $b_\phi$ defined on $V$ by:
$$b_\phi (\omega)= \int_{\partial M} \<  \phi, \iota^* \omega \> \, d \mu_g, \quad \forall \omega \in V,$$
is obviously continuous on $V$. Hence, for all $\phi \in L^2(\partial M)$, the Lax-Milgram Lemma ensures the existence of a unique $\eta=\eta_\phi \in V$ such that
\begin{equation}
    \label{variat_sol_Lax_Milgram}
    a_\beta (\eta_\phi, \omega) =b_\phi(\omega), \quad \forall \omega \in V.
\end{equation}
Furthermore, the map $I: \phi \longmapsto \eta_\phi$ is linear, continuous and injective from $L^2(\partial M)$ to $V$. To see the injectivity, assume $\eta_\phi=0$ for some $\phi \in L^2(\partial M)$, then for every $\omega \in V$, $\int_{\partial M} \<  \phi, \iota^* \omega \> \, d \mu_g=0$. Now, for any $\psi \in \Omega^p(\partial M)$, the harmonic tangential extension of $\psi$ belongs to $V$, so $\Omega^p(\partial M) \subset \iota^*(V)$. As a consequence, we deduce that $\iota^*(V)$ is dense in $L^2(\partial M)$. It follows that $\phi=0$. We then define the map $K : V \longrightarrow V$ by $K:= I \circ T \circ i$, where $T : H^1(M) \rightarrow L^2(\partial M)$ is the trace map and $i$ is the inclusion map of $V$ into $H^1(M)$. In particular, $K$ is injective as a consequence of $I$ being injective. Remark that $K(\omega)=\eta_{\iota^* \omega}$
for $\omega \in V$. 
\begin{prop}
    The operator $K$ is compact, self-adjoint and positive definite with respect to  $a_\beta$. 
\end{prop}
{\it Proof.} First, $K$ is compact because the inclusion from $H^{\frac{1}{2}}(\partial M)$ to $L^2(\partial M)$ is compact according to the Rellich-Kondrachov theorem. Moreover, $K$ is self-adjoint with respect to $a_\beta$. Indeed, by \eqref{variat_sol_Lax_Milgram}, we have for $\omega, \omega' \in V$  
\begin{equation}
    \label{K}
a_\beta(K(\omega), \omega')
=\int_{\partial M} \<  \iota^* \omega, \iota^* \omega'\> \, d \mu_g.\end{equation}
The previous identity immediately implies that $a_\beta(K(\omega),\omega) \ge 0$ for all $\omega$, with equality if and only if $\iota^* \omega=0$. In particular, equality yields  $K(\omega)=0$, and then the injectivity of $K$ ensures that $\omega=0$. 
\hfill$\square$ 

\vspace{0.3em}

In the following, we shall use the properties of the operator $K$ introduced above to study the spectrum of the Wentzell operator by relating it to that of $K$. We note that the Wentzell operator can be defined on the space $\iota^* V$.
\begin{thm}\label{thm:wentbc} Let $(M^{n+1},g)$ be a compact Riemannian manifold with smooth boundary $\partial M$ and let $\nu$ be the inward unit  vector field normal to the boundary. Let $\beta$ be a positive real number. Then, for $p=0,1, \dots, n$, 
the problem $\mathcal{L}^{[p]}_\beta  \omega =\lambda \omega$ has a discrete spectrum consisting of an unbounded nondecreasing sequence of nonnegative eigenvalues of finite multiplicities $$\lambda_{1,p, \beta}(M) \le \lambda_{2,p,\beta}(M) \le \dots.$$
Its kernel is given by $\mathcal{H}^p(\partial M) \cap \iota^* (\mathcal{H}^p_N(M))$. Moreover, the first eigenvalue $\lambda_{1,p, \beta}(M)$ of the Wentzell-Laplace operator is characterized as follows:
      \begin{equation}
          \label{variat_char}         \lambda_{1,p,\beta}(M)=\inf \left\{ \dfrac{\|d \omega\|^2_{L^2(M)}+\|\delta \omega\|^2_{L^2(M)}+\beta \, \left(\|d^{\partial M} \iota^* \omega\|^2_{L^2(\partial M)}+\|\delta^{\partial M} \iota^* \omega\|^2_{L^2(\partial M)}\right)}{\|\iota^* \omega\|^2_{L^2(\partial M)}} \right\},
      \end{equation}
      where the infimum is taken on all $p$-forms $\omega \in \Omega^p(M)\setminus\{0\}$ such that $\nu \lrcorner\, \omega =0$. This infimum is attained precisely when $\iota^*\omega$ is an eigenform of \eqref{wentzell_operator} associated with $\lambda_{1,p,\beta}(M)$.
\end{thm}
{\it Proof.} 
As the operator $K$ given by \eqref{K} is a compact, self-adjoint and positive definite operator on the Hilbert space $V$ equipped with the inner product $a_\beta$, the spectral theorem ensures the existence of a Hilbert orthonormal basis $(\omega_i)_{i\ge 1}$ of $V$ and a nonincreasing sequence of positive real numbers $(\mu_i)_{i\ge 1}$ satisfying $\mu_i \to 0$ as $i \to \infty$, and such that $K\omega_i = \mu_i \omega_i$ for each $i\ge 1$. \\\\
For each $i \ge 1$ and 
$\omega \in V$, we have by \eqref{K}
\begin{equation}
\label{eigen_K}
    a_\beta(\omega_i, \omega)=  \dfrac{1}{\mu_i}\, ( \iota^* \omega_i , \iota^* \omega )_{L^2(\partial M)}.
\end{equation}
Therefore, 
\begin{equation}
    \label{eq:eigweakW}
\mathcal{D}(\omega_i,
\omega)+\beta\,\mathcal{D}^{\partial M} (\iota^*\omega_i,\iota^*\omega)
=\left(\frac{1}{\mu_i}-1\right)\left(\iota^*\omega_i,\iota^*\omega\right)_{L^2(\partial M)} \quad \forall \, \omega \in V.\end{equation}

In addition, using Equation \eqref{partial_integration_laplacian}, we have for every $\omega\in Z$,
\begin{eqnarray}\label{eq:Deltaomegaiomega}\medskip
\nonumber\mathcal{D}(\omega_i,
\omega)&=&\big(\underbrace{\Delta\omega_i}_{0},\omega\big)_{L^2(M)}-\big(\nu\lrcorner \,
d\omega_i,\iota^*\omega\big)_{L^2(\partial
M)}+\big(\iota^*\delta\omega_i,\underbrace{\nu\lrcorner\, \omega}_{0}\big)_{L^2(\partial M)}\\ \medskip
&=&-\left(\nu\lrcorner \,
d\omega_i,\iota^*\omega\right)_{L^2(\partial
M)}.
\end{eqnarray}
Also, the identity
\begin{eqnarray}\label{eq:Deltaboundomegaiomega}
\mathcal{D}^{\partial M}(\iota^*\omega_i,
\iota^*\omega)&=&\left(\Delta^{\partial M}(\iota^*\omega_i),\iota^*\omega\right)_{L^2(\partial M)},
\end{eqnarray}
holds for every $\omega \in Z$.
Note that $\omega_i \in H^{3/2}(M)$ by the elliptic regularity of the Dirichlet boundary conditions for the Poisson equation,  so we see $\Delta\omega_i$ and $\Delta^{\partial M}(\iota^*\omega_i)$ as
elements in $H^{-1/2}(M)$ and $H^{-1}(\partial M)$ respectively.
Substituting \eqref{eq:Deltaomegaiomega} and \eqref{eq:Deltaboundomegaiomega} into \eqref{eq:eigweakW}, we deduce
$$-\left(\nu\lrcorner \,
d\omega_i,\iota^*\omega\right)_{L^2(\partial
M)}+ \beta \, \left(\Delta^{\partial M}(\iota^*\omega_i),\iota^*\omega\right)_{L^2(\partial M)}=\left(\frac{1}{\mu_i}-1\right)\left(\iota^*\omega_i,\iota^*\omega\right)_{L^2(\partial M)}
$$
for all $\omega\in Z$. 
Since the map $Z\to \Omega^p(\partial M)$,
$\omega\mapsto \iota^* \omega$ is continuous w.r.t.
$\|\cdot\|_V$ and
$\|\cdot\|_{L^2(\partial M)}$ and is bijective, we conclude that $\iota^* \omega_i$ is an eigenform of the Wentzell problem associated with the eigenvalue $\lambda_i=1-\frac{1}{\mu_i}$. 
Note that $\omega_i$ is smooth as the operator $\mathcal{L}^{[p]}_\beta$ is an elliptic pseudo-differential operator, being the sum of a first-order pseudo-differential operator and an elliptic differential operator of order 2 \cite[Chap. 7, Prop. 4.1]{TaylorPDEII}. \\\\
Conversely, let $\lambda$ be a nonnegative real number and $\phi$ a $p$-form on $\partial M$ satisfying  \eqref{wentzell_operator}.  
For $\psi \in \Omega^p(\partial M)$, denote by $\hat{\phi}$ and $\hat{\psi}$ the harmonic tangential extensions of $\phi$ and $\psi$, respectively. 
Applying Equation \eqref{partial_integration_laplacian} to $\hat{\phi}$ and $\hat{\psi}$, we obtain
$$- \int_{\partial M} \<  \nu \lrcorner \, d \hat{\phi} , \psi \> \, d \mu_g=\mathcal{D}(\hat{\phi} , \hat{\psi}) .$$
Similarly, the partial integration formula in $\partial M$ gives 
$$\int_{\partial M} \<  \Delta^{\partial M} \phi , \psi \> \, d \mu_g=\mathcal{D}^{\partial M}({\phi} , {\psi}) .$$
Substituting the previous equalities in the expression of $\mathcal{L}^{[p]}_\beta$ yields
\begin{eqnarray}
\label{wentselfadjoint}
    \int_{\partial M} \<  \mathcal{L}^{[p]}_\beta  \phi , \psi \> \, d \mu_g 
      = \mathcal{D}(\hat{\phi} , \hat{\psi})  
       + \beta \, \mathcal{D}^{\partial M}({\phi} , {\psi}). 
\end{eqnarray}
Thus, 
$\mathcal{D}(\hat{\phi}, \hat{\psi}) + \beta \, \mathcal{D}^{\partial M} (\phi, \psi)=\lambda \, (\phi, \psi)_{L^2(\partial M)}$,
 or equivalently $a_\beta(\hat{\phi}, \hat{\psi})=(\lambda+1)\, (\phi, \psi)_{L^2(\partial M)}$. Using once again the continuity and bijectivity of the map $\iota^* : (Z, \|.\|_V) \rightarrow (\Omega^p(\partial M), \|.\|_{L^2(\partial M)})$, we find that  
 $a_\beta(\hat{\phi},\omega)= (\lambda+1) ( \phi , \iota^* \omega)_{L^2(\partial M)}$ for all $\omega\in V$. It follows from \eqref{K} that $a_\beta(\hat{\phi},\omega)= (\lambda+1) \, a_\beta(K(\hat{\phi}),\omega)$ for all $\omega \in V$ which means that $\hat{\phi}$ is an eigenform of $K$ associated with the eigenvalue $\dfrac{1}{\lambda +1}$. We have thus proved the first part of the theorem.
\\\\
Next, we determine the kernel of $\mathcal{L}^{[p]}_\beta$. Setting $\psi=\phi$ in \eqref{wentselfadjoint} gives: 
\begin{equation}
\label{eq:Lphiphi}
      \int_{\partial M} \<  \mathcal{L}^{[p]}_\beta  \phi , \phi \> \, d \mu_g =\mathcal{D}(\hat{\phi} , \hat{\phi})  
       + \beta \, \mathcal{D}^{\partial M}({\phi} , {\phi}) \quad \forall \, \phi \in \Omega^p(\partial M).
\end{equation}
Then, a $p$-form $\phi$ belongs to the kernel of $\mathcal{L}^{[p]}_\beta$ if and only if $d^{\partial M} \phi= \delta^{\partial M} \phi =0$ and $d \hat{\phi}=\delta \hat{\phi}=0$. The desired equality then follows, as $\nu \lrcorner \, \hat{\phi}=0$.
\\\\
%
%
%
%
%
%
%
To prove the last part of the theorem, we consider the space $H$ defined by \eqref{H}. 
It is straightforward that the inner product $(.,.)_V$ is well defined on $H$. Moreover, $H$ is a Hilbert space with respect to this inner product. It is also clear that $a_\beta$ defines an inner product on $H$. We need the following lemma:
\begin{lemma}
\label{lem_splitting}
    The Hilbert space $H$ admits the following $a_\beta$-orthogonal direct sum decomposition:
    \begin{equation}
        \label{dec_H}
        H=V\oplus H^1_0(M).
    \end{equation}
\end{lemma}
{\it Proof.} Let $\omega \in H^1_0(M)$ and $\omega' \in  V$. Since $\Delta \omega' =0$ weakly on $M$, we have  $\mathcal{D}(\omega, \omega')=0$, yielding $$a_\beta(\omega, \omega')=\underbrace{\mathcal{D}(\omega,\omega')}_{0} + \beta \, \mathcal{D}^{\partial M}(\underbrace{\iota^* \omega}_{0},\iota^*\omega') + (\underbrace{\iota^* \omega}_{0},\iota^*\omega')_{L^2(\partial M)}=0.$$ 
This shows that $H^1_0(M) \subset V^\perp$. 
Conversely, let $\omega \in V^\perp$, then for each $i \ge 1$, $a_\beta (\omega_i, \omega)=0$. Therefore, for each $i \ge 1$, $\left(\mathcal{L}_\beta^{[p]} (\iota^* \omega_i), \iota^* \omega \right)_{L^2(\partial M)}+\left(\iota^* \omega_i, \iota^* \omega \right)_{L^2(\partial M)}=0$. It follows that, $(\lambda_{i,p,\beta}+1)\left(\iota^* \omega_i, \iota^* \omega \right)_{L^2(\partial M)}=0$. Thus, $\left(\iota^* \omega_i, \iota^* \omega\right)_{L^2(\partial M)}=0 $ for each $i \ge 1$, which implies that $\iota^* \omega'=0$. In addition, knowing that $\nu \lrcorner\, \omega=0$ because $\omega\in H$, we deduce that $\omega$ vanishes along the boundary, so it belongs to $H_0^1(M)$. 
This concludes the proof of the lemma.
\hfill$\square$

\vspace{0.3em}

%
%
We end the proof of Theorem \ref{thm:wentbc}. For every eigenform $\phi$ associated with a nonnegative eigenvalue $\lambda$ to the Wentzell operator, using \eqref{eq:Lphiphi}, we obtain $\mathcal{D}(\hat{\phi},\hat{\phi})+\beta \, \mathcal{D}^{\partial M}( \phi,\phi)=\lambda \, \|\phi \|_{L^2(\partial M)}^2$. Therefore, renaming the form $\hat{\phi}$ as $\omega$, where $\phi$ is an eigenform of the Wentzell operator associated with the eigenvalue $\lambda_{1,p,\beta}(M)$, we have 
\begin{equation}
\label{eigen_eigen}
     \lambda_{1,p,\beta}(M) = \dfrac{\mathcal{D}(\omega,\omega)+\beta \, \mathcal{D}^{\partial M}( \iota^* \omega,  \iota^* \omega)}{\|\iota^* \omega\|^2_{L^2(\partial M)}}.
\end{equation}
On the other hand, by the min-max principle, the first eigenvalue $\mu_1$ of $K$ is characterized by
$\mu_1=\sup_{\omega \in V\setminus\{0\}}\dfrac{a_\beta(K(\omega),\omega)}{a_\beta(\omega, \omega)}$.
Therefore, using \eqref{K}, we deduce that
$\mu_1 \ge \dfrac{\|\iota^* \omega\|^2_{L^2(\partial M)}}{a_\beta(\omega,\omega)}$,
for all $\omega \in V\setminus \{0\}$. Taking into account that $\lambda_{1,p,\beta}(M)=\dfrac{1}{\mu_1}-1$ obtained in the proof of the first part of the theorem, we deduce the inequality $\lambda_{1,p,\beta}(M) \le \dfrac{a_\beta(\omega,\omega)}{\|\iota^* \omega\|^2_{L^2(\partial M)}}-1$ for all nonzero $\omega$ in $V$, that is
\begin{equation}
    \label{ineq:eigen}
\lambda_{1,p,\beta}(M) \le \dfrac{\mathcal{D}(\omega,\omega)+\beta \, \mathcal{D}^{\partial M}( \iota^* \omega,  \iota^* \omega)}{\|\iota^* \omega\|^2_{L^2(\partial M)}} \quad \forall \, \omega \in V\setminus\{0\}.\end{equation}
Combining \eqref{eigen_eigen} and \eqref{ineq:eigen}, we get
      \begin{equation*}
               \lambda_{1,p,\beta}(M)=\inf_{\omega \in V\setminus\{0\}}  \dfrac{\mathcal{D}(\omega,\omega)+\beta \, \mathcal{D}^{\partial M}( \iota^* \omega,  \iota^* \omega)}{\|\iota^* \omega\|^2_{L^2(\partial M)}},
      \end{equation*}
      which is actually an infimum on $Z\setminus\{0\}$ because the eigenforms are smooth. Furthermore, 
the condition $\Delta \omega=0$ weakly on $M$ can be dropped, that is, the infimum can be taken on the smooth nonzero forms belonging to $H$. Indeed, let $\omega$ be a nonzero form in $H$. By Lemma \ref{lem_splitting}, $\omega=\omega_V + \omega_0 \in V \oplus H^1_0(M)$ and $a_\beta(\omega, \omega)=a_\beta(\omega_V,\omega_V) + a_\beta(\omega_0,\omega_0) \ge a_\beta(\omega_V, \omega_V)$. Therefore,
\begin{eqnarray*}
    \mathcal{D}(\omega,\omega)+\beta \, \mathcal{D}^{\partial M} (\iota^* \omega, \iota^* \omega) +\|\iota^* \omega\|_{L^2(\partial M)}^2 & \ge & \mathcal{D}(\omega_V,\omega_V)+\beta \, \mathcal{D}^{\partial M} (\iota^* \omega_V, \iota^* \omega_V) +\|\iota^* \omega_V\|_{L^2(\partial M)}^2.
\end{eqnarray*}
Using the fact that $\iota^* \omega=\iota^* \omega_V$, we deduce the result.     \hfill$\square$
\begin{remark}
    As the space $\{\iota^* \omega \mid \omega \in V\}$ is dense in $L^2(\partial M)$, we deduce from  \eqref{eigen_K} that the sequence $(\sqrt{\mu_i}\, \, \iota^* \omega_i)_{i \ge 1}$ is a Hilbert orthonormal basis of $L^2(\partial M)$, made of eigenforms of the Wentzell operator \eqref{wentzell_operator}.
\end{remark}
\begin{remark}
     The variational  formulation of the Wentzell eigenvalue problem \eqref{wentzell_operator} is expressed as: \\
 Find $(\lambda,\omega) \in \R \times (H\setminus \{0\})$ such that
  \begin{equation}
     \label{weak_solution_Wen}
     \mathcal{D}(\omega,\omega') + \beta \, \mathcal{D}^{\partial M}(\iota ^*\omega, \iota ^* \omega')  
   = \lambda \, \left(\iota ^* \omega,  \iota^* \omega'\right)_ {
L^2(\partial M)} \quad \forall \, \omega'\in H.
   \end{equation}
Indeed, we first note that any $\omega$ that satisfies \eqref{weak_solution_Wen} belongs to $V$ in view of Lemma \ref{lem_splitting}. In fact, from the proof of Theorem \ref{thm:wentbc}, the required identity holds for all $\omega' \in V$. Therefore, we conclude using once again the splitting \eqref{dec_H}.
\end{remark}
\begin{remark}
    In the case $p=0$, 
    the first positive eigenvalue of $\mathcal{L}^{[0]}_\beta=\mathcal{L}_\beta$  is $\lambda_{2,0,\beta}(M)$, according to our convention.
\end{remark}
We conclude this section by defining the dual Wentzell operator, following the same approach as that used in \cite{RS_DtN2012} to define the dual Dirichlet-to-Neumann operator. 
\begin{definition}
\label{def:dual_went}
    The dual Wentzell operator on $p$-forms on the boundary is defined as:
    \begin{equation}
        \label{dual_went}
        \mathcal{L}_{\beta,D}^{[p]}=(-1)^{p(n-p)} \star_{\partial M} \mathcal{L}_{\beta}^{[n-p]} \star_{\partial M}.
    \end{equation}
    Here $\star_{\partial M}$ denotes the Hodge star operator on the boundary.
\end{definition}
The dual operator defined above is equal to $\mathcal{T}^{[p]}_D +\beta \, \Delta^{\partial M}$, where $\mathcal{T}^{[p]}_D$ is the dual Dirichlet-to-Neumann operator defined as follows \cite{RS_DtN2012}: for $\phi \in \Omega^p(\partial M)$, $\mathcal{T}^{[p]}_D(\phi)=\iota^* (\delta \tilde{\phi})$, where  $\tilde{\phi}$ is the unique harmonic normal extension of $\phi$, that is, $\tilde{\phi}$ is the $(p+1)$-form on $M$ satisfying
$$
\left\{\begin{array}{lll}\Delta \tilde{\phi}&=0&\textrm{ on
}M\\ \iota^* \tilde{\phi}&=0&\textrm{ on
}\partial M\\
\nu \lrcorner \, \tilde{\phi}&=\phi &\textrm{ on
}\partial M.
\end{array}
\right.
$$
Thus, $\mathcal{L}_{\beta,D}^{[p]} (\phi)=\iota^* (\delta \tilde{\phi}) +\beta \, \Delta^{\partial M} \phi$. 
The eigenvalues $\lambda_{k,p,\beta}^D(M)$ of the dual Wentzell eigenvalue problem are given by
$$\lambda_{k,p,\beta}^D(M)=\lambda_{k,n-p,\beta}(M).$$
Moreover, the first eigenvalue is characterized by:
\begin{equation}
    \label{var_char_dual}
    \lambda_{1,p,\beta}^D(M)= \inf \left\{ \dfrac{\|d \omega\|^2_{L^2(M)}+\|\delta \omega\|^2_{L^2(M)}+\beta \,\left(\|d^{\partial M} (\nu \lrcorner\, \omega)\|^2_{L^2(\partial M)}+\|\delta^{\partial M} (\nu \lrcorner\, \omega)\|^2_{L^2(\partial M)}\right)}{\|\nu \lrcorner\, \omega\|^2_{L^2(\partial M)}} \right\},
\end{equation}
 where the infimum is taken on all $p$-forms $\omega \in \Omega^{p+1}(M)\setminus\{0\}$ such that $\iota^* \omega =0$. Moreover, this infimum is attained exactly for the eigenforms of \eqref{dual_went} associated with $\lambda_{1,p,\beta}^D(M)$. 

\begin{remark}
    When $\beta=0$, the variational characterizations \eqref{variat_char} and \eqref{var_char_dual} reduce to the corresponding characterizations of the Dirichlet-to-Neumann operator and its dual (see identities (7) and (9) in \cite{RS_DtN2012}).
\end{remark}


\section{The spectrum of the Wentzell operator on the Euclidean ball}
\label{Sec:Wentz_spectrum_Euclidean_ball}
In this section, we give the eigenvalues of the Wentzell operator acting on $p$-forms on the unit ball $\mathbb{B}^{n+1}$ in $\R^{n+1}$. They are directly obtained from the eigenvalues of the Dirichlet-to-Neumann operator on the ball performed in \cite{RS_DtN2014} and those of the Hodge Laplacian on the sphere \cite{GM}. The result is based on the fact that these two operators have the same eigenforms in each degree. We  state the result separately for the cases $p=0$, $1 \le p \le n-1$, and $p=n$. We prove the result for $p \ge 1$. The case $p=0$ is well known \cite{DKL2016}, and we give its statement for completeness.
\\
 For $k=0,1,2, \dots$, we denote 
by $H_{k,0}$ the vector space of  homogeneous polynomial harmonic functions of degree $k$, and let $M_{k,0}:=\dim H_{k,0}$. 
\begin{thm}
\label{spectrum_scal_Went_ball}
    The eigenvalues of the Wentzell operator $\mathcal{L}_\beta^{[0]}$ on the unit ball of $\R^{n+1}$ are given as follows:
    $$\lambda_{k, 0,\beta }''(\mathbb{B}^{n+1})= k-1 + \beta \, (k-1)(n+k-2), \quad k=1,2,\dots $$
        with multiplicity $M_{k-1,0}$ where $k \ge 1$. The eigenspace associated with $\lambda_{k, 0,\beta }''(\mathbb{B}^{n+1})$ is $\iota^* (H_{k-1,0})$. 
\end{thm}
Assume now $p \ge 1$. 
Recall that a $p$-form $\omega =\displaystyle \sum_{1 \le i_1 < \dots < i_p \le n+1} \, \omega_{i_1 \dots i_p} \, dx_{i_1} \wedge \dots \wedge dx_{i_p}$ on $\R^{n+1}$ is called polynomial of homogeneous degree $k$ if the functions $\omega_{i_1\dots i_p}$ are polynomial of homogeneous degree $k$. 
Let $P_{k,p}$ be the vector space of polynomial $p$-forms of homogeneous degree $k \ge 0$ on $\R^{n+1}$. For each $k=0,1,2, \dots$ and $p=1, \dots , n$, we consider the following spaces:
\begin{eqnarray*}
    H_{k,p}&:=&\{\omega \in P_{k,p} \mid \Delta \omega=0, \delta \omega =0\}, \\
    H_{k,p}'&:=&\{\omega \in H_{k,p} \mid d \omega=0\},\\
    H_{k,p}''&:=&\{\omega \in H_{k,p} \mid Z \lrcorner\, \omega=0\},
    \end{eqnarray*}
where $Z=\sum_{i=1}^{n+1} x_i \dfrac{\partial}{\partial x_i}$ is the radial vector field. We also denote by $M_{k,p}$ the dimension of $H_{k,p}''$.
\begin{thm}
\label{spectrum_1_n-1_Went_ball}
    Let $1 \le p \le n-1$. The spectrum of the Wentzell operator $\mathcal{L}^{[p]}_\beta$ on the unit ball of $\R^{n+1}$ consists of the following families of eigenvalues, idexed by $k=1,2, \dots$:
    $$
    \left\{
\begin{array}{ll}\medskip
    \lambda_{k, p,\beta }''(\mathbb{B}^{n+1})= k+p + \beta \, (k+p)(n+k-p-1),  \\\medskip
     \lambda_{k, p,\beta}'(\mathbb{B}^{n+1})= (k+p-1) \dfrac{n+2k+1}{n+2k-1} + \beta \, (k+p-1)(n+k-p).
\end{array}
    \right.
    $$
    The eigenvalues $\lambda_{k, p,\beta }''(\mathbb{B}^{n+1})$ and $\lambda_{k, p,\beta}'(\mathbb{B}^{n+1})$ have multiplicities $M_{k,p}$ and $M_{k,p-1}$, and eigenspaces $\iota^* (H_{k,p}'')$ and $\iota^* (H_{k-1,p}')$, respectively. 
    \end{thm}
{\it Proof.} 
The eigenvalues and the corresponding eigenspaces of the Hodge Laplacian on the sphere are \cite{GM}:
$$
\left\{
\begin{array}{lll}
\medskip
    \lambda_{k,p}''(\mathbb{S}^n) =(k+p)(n+k-p-1) & \textrm{with multiplicity } M_{k,p} \textrm{ and eigenspace } \iota^*(H_{k,p}''), \\
    \medskip
     \lambda_{k,p}' (\mathbb{S}^n)=(k+p-1)(n+k-p) & \textrm{with multiplicity } M_{k,p-1} \textrm{ and eigenspace } \iota^*(H_{k-1,p}'),
\end{array}
\right. 
$$
for each $k=1,2, \dots$. 
In addition, the eigenvalues of the Dirichlet-to-Neumann operator on the unit ball of $\mathbb{R}^{n+1}$ are given as follows \cite{RS_DtN2014}:
$$
\left\{
\begin{array}{lll}
\medskip
    \nu_{k,p}''(\mathbb{B}^{n+1}) =k+p \,\,\,\, \textrm{with multiplicity } M_{k,p} \textrm{ and eigenspace } \iota^*(H_{k,p}''),  \\
    \medskip
     \nu_{k,p}' (\mathbb{B}^{n+1})=(k+p-1)\dfrac{n+2k+1}{n+2k-1} \quad \textrm{with multiplicity } M_{k,p-1} \textrm{ and eigenspace } \iota^*(H_{k-1,p}'),
\end{array}
\right. 
$$
for each $k=1,2, \dots$. 
\\
Therefore, it follows directly that the tangent part of any 
form belonging to $H_{k,p}''$ is a co-exact eigenform of the Wentzell operator $\mathcal{L}^{[p]}_\beta$ associated with the eigenvalue $k+p+ \beta \, (k+p)(n+k-p-1)$.  Similarly, any form of $\iota^* (H_{k,p-1}')$ is an exact eigenform of the Wentzell operator $\mathcal{L}^{[p]}_\beta$ associated with the eigenvalue $(k+p-1)\left(n+k-p+ \beta \,\frac{n+2k+1}{n+2k-1}\right)$. 
Moreover, since the direct sum of the eigenspaces of $\mathcal{L}^{[p]}_\beta$ is dense in $L^2(\Omega^p(\mathbb{S}^n))$, we deduce that the values given above are precisely all the eigenvalues of the Wentzell operator $\mathcal{L}^{[p]}_\beta$. \hfill$\square$

\vspace{0.3em}

Finally, for $p=n$, the eigenvalues of the Hodge Laplacian on $\mathbb{S}^n$ are given by:
$$\left\{
\begin{array}{lll}
\medskip
    \lambda_{1,n}''(\mathbb{S}^n) =0 \quad \textrm{with multiplicity one},  \\
    \medskip
     \lambda_{k,n}'(\mathbb{S}^n) =k(k+n-1) \quad \textrm{with multiplicity } M_{k,n-1} \,\, (k=1,2,\dots).
\end{array}
\right.$$
The volume form of $\mathbb{S}^n$ is an eigenform associated with $\lambda_{1,n}''(\mathbb{S}^n)$.  The eigenspace associated with $\lambda_{k, n}'(\mathbb{S}^n)$ is $\iota^* (H_{k-1,n}')$ which consists of exact $n$-forms.
\\\\
The eigenvalues of the Dirichlet-to-Neumann operator $\mathcal{T}^{[n]}$ on the ball $\mathbb{B}^{n+1}$ are as follows:
$$\left\{
\begin{array}{lll}
\medskip
    \nu_{1,n}''(\mathbb{B}^{n+1}) =n+1 \quad \textrm{with multiplicity one},  \\
    \medskip
     \nu_{k,n}'(\mathbb{B}^{n+1})=(k+n-1) \dfrac{n+2k+1}{n+2k-1}  \quad \textrm{with multiplicity } M_{k,n-1} \,\, (k=1,2,\dots).
\end{array}
\right.$$
The volume form of $\mathbb{S}^n$ is an eigenform associated with $\nu_{1,n}''(\mathbb{B}^{n+1})$.  The eigenspace associated with $\nu_{k, n}'(\mathbb{B}^{n+1})$ is $\iota^* (H_{k-1,n}')$ which consists of exact $n$-forms.
\\\\
As in Theorem \ref{spectrum_1_n-1_Went_ball}, we immediately deduce the spectrum of the Wentzell operator $\mathcal{L}^{[n]}_\beta$ on the unit ball of $\mathbb{R}^{n+1}$. 
\begin{thm}
\label{spectrum_n_Went_ball}
    The spectrum of the Wentzell operator $\mathcal{L}^{[n]}_\beta$ on the unit ball of $\R^{n+1}$ consists of the following eigenvalues:
    \begin{enumerate}
            \item $\lambda_{1,n,\beta}''(\mathbb{B}^{n+1})= n+1$,  
        with multiplicity one and associated eigenform the volume form of $\mathbb{S}^n$.
       \item  $\lambda_{k,n, \beta}'(\mathbb{B}^{n+1})= (k+n-1) \dfrac{n+2k+1}{n+2k-1} + \beta \, k (k+n-1)$, 
        with multiplicity $M_{k,n-1}$ where $k \ge 1$. The eigenspace associated with $\lambda_{k, n,\beta}'(\mathbb{B}^{n+1})$ is $\iota^* (H_{k-1,n}')$ which consists of exact $n$-forms.  
    \end{enumerate}
\end{thm}
We conclude this section with the expression of the smallest eigenvalue of the Wentzell operator on the Euclidean unit ball. When $p=0$, the smallest eigenvalue is 0 and for $p=n$, the first eigenvalue is $n+1$ with multiplicity 1. 
\begin{cor}
\label{cor:first_eigen_ball}
    Let $1 \le p \le n-1$. The first eigenvalue $\lambda_{1,p,\beta}(\mathbb{B}^{n+1})$ of $\mathcal{L}^{[p]}_\beta$ is given by:
    $$\lambda_{1,p,\beta}(\mathbb{B}^{n+1})=\left\{
    \begin{array}{cc}
    \medskip
        p\left(\dfrac{n+3}{n+1} +\beta \, (n+1-p)  \right) &  \textrm{if } p <  \dfrac{n+1}{2}, \\
        \medskip
       (p+1)(1+\beta \, (n-p))  & \textrm{if } p \ge  \dfrac{n+1}{2}.  
    \end{array}
    \right.
    $$
\end{cor}
 {\it Proof.}  A simple computation gives $$\lambda_{1,p,\beta}'(\mathbb{B}^{n+1})-\lambda_{1,p,\beta}''(\mathbb{B}^{n+1})= \dfrac{2p-n-1}{n+1}+\beta\,(2p-n).$$
Since the above term is positive if and only if $p \ge \frac{n+1}{2}$, this completes the proof.  \hfill$\square$

\section{Eigenvalue estimates}
\label{Sec:Wentz_eigen_estimates}
Let $M$ be a compact Riemannian manifold with smooth boundary and $\beta$ a positive real number. We will give different estimates of the first eigenvalue of the Wentzell operator.  

The following result generalizes inequality (1.6) in \cite{XiaWang2018} to differential forms.
\begin{prop}
\label{min_Wentz_DtN_Delta_bound}
We have the following estimate for the first eigenvalue of the Wentzell operator $\mathcal{L}_\beta^{[p]}$ involving the first eigenvalues $\nu_{1,p}(M)$ of the Dirichlet-to-Neumann operator and $\lambda_{1,p}(\partial M)$ of the Hodge Laplacian on $\partial M$: 
 \begin{equation}
     \label{est_Wen_boundLapl_DtN}
  \lambda_{1,p,\beta}(M) \ge \nu_{1,p}(M)+  \beta \, \lambda_{1,p}(\partial M) ,   
 \end{equation} 
 with equality in the case of a Euclidean ball. 
\end{prop} 
{\it Proof. } Let $\phi$ be an eigenform of $\mathcal{L}^{[p]}_\beta$ associated with the eigenvalue $\lambda_{1,p,\beta}(M)$. We denote by $\hat{\phi}$ the unique harmonic tangential extension of $\phi$ to $M$. As $\nu \lrcorner \, \hat{\phi}=0$, by the variational characterization of $\nu_{1,p}(M)$ \cite[Thm. 11]{RS_DtN2012}, we have the inequality:
\begin{eqnarray*}
    \|\phi\|_{L^2(\partial M)}^2 \, \nu_{1,p}(M) \le \|d \hat{\phi}\|_{L^2(M)}^2 + \|\delta \hat{\phi}\|_{L^2(M)}^2.
\end{eqnarray*}
Similarly, 
\begin{eqnarray*}
    \|\phi\|_{L^2(\partial M)}^2 \, \lambda_{1,p}(\partial M) \le \|d^{\partial M} {\phi}\|_{L^2(\partial M)}^2 + \|\delta^{\partial M} {\phi}\|_{L^2(\partial M)}^2.
\end{eqnarray*}
Therefore, 
\begin{eqnarray*}
    \|\phi\|_{L^2(\partial M)}^2 \left(\nu_{1,p}(M) +\beta \,\lambda_{1,p}(\partial M) \right) \le \|d \hat{\phi}\|_{L^2(M)}^2 + \|\delta \hat{\phi}\|_{L^2(M)}^2 + \beta \, \left( \|d^{\partial M} {\phi}\|_{L^2(\partial M)}^2 + \|\delta^{\partial M} {\phi}\|_{L^2(\partial M)}^2\right).
\end{eqnarray*}
After division by $\|\phi\|_{L^2(\partial M)}^2$, the r.h.s. reduces to $\lambda_{1,p,\beta}(M)$. Thus, we obtain the required inequality. When $M$ is a Euclidean ball, we can easily check that equality is achieved in \eqref{est_Wen_boundLapl_DtN} using the eigenvalues of $\mathcal{L}^{[p]}_\beta$ computed above, together with those of the Dirichlet-to-Neumann operator and the boundary Hodge Laplacian. \hfill$\square$
\begin{remark}
   Equality holds in \eqref{est_Wen_boundLapl_DtN} if and only if  every  eigenform of $\mathcal{L}^{[p]}_\beta$ associated with the eigenvalue $\lambda_{1,p,\beta}(M)$ is simultaneously an eigenform of the Dirichlet-to-Neumann operator associated with $\nu_{1,p}(M)$ and an eigenform of the boundary Hodge Laplacian associated with $\lambda_{1,p}(\partial M)$.
\end{remark}
%
Let us now recall some notions that will be used in what follows. We denote by $W^{[p]}$ the Bochner operator 
$W^{[p]}:=-\sum_{i,j=1}^{n+1} e_i^* \wedge(e_j\lrcorner \, R(e_i,e_j))$, 
where $\{e_i\}_{i=1, \dots ,n+1}$ is
a local orthonormal frame of $TM$ and $R$ is the Riemann curvature tensor of $(M,g)$ given by $R(X,Y)=[\nabla_X, \nabla_Y]-\nabla_{[X,Y]}$  for all $X,Y \in \Gamma(TM)$.
 We also denote by $S$ the shape operator given by $S(X)=-\nabla_X \nu$ for $X \in \Gamma(T\partial M)$,  and by $S^{[p]}$ its canonical extension to $p$-forms defined by
\begin{equation}\label{ext_S}
S^{[p]}(\omega)(X_1, \dots , X_p):=\sum_{j=1}^p \omega (X_1, \dots , S(X_j), \dots , X_p),\end{equation}
for $\omega \in \Omega^p(\partial M)$ and $X_1, \dots, X_p \in \Gamma(T\partial M)$. For $p=0$, by convention, $S^{[0]}=0$. If $\eta_1\le  \dots \le \eta_{n}$ are the principal curvatures of $\partial M$, set $\sigma_p=\eta_1 + \dots + \eta_p$, and $\sigma_p(\partial M)=\displaystyle\inf_{x \in \partial M} \sigma_p(x)$. \\\\
Using Proposition~\ref{min_Wentz_DtN_Delta_bound}, we derive the following estimates.
\begin{prop}
\label{lower_bounds_wentzell}
    Let $p=1, \dots , n-1$. Assume that $W^{[p]}\ge 0$, $W^{[p+1]}\ge 0$. Then 
    \begin{enumerate}
        \item If $1 \le p < \frac{n+1}{2}$, then under the condition  $\sigma_p(\partial M) >0$, we have \begin{equation}\label{lower_bound_Went_first_case}\lambda_{1,p,\beta} (M) > \frac{n-p+2}{n-p+1}\sigma_p(\partial M) + \beta \, \min(\sigma_p(\partial M) \sigma_{n-p+1}(\partial M),\sigma_{p+1}(\partial M) \sigma_{n-p}(\partial M) ).\end{equation}
        \item If $p \ge \frac{n+1}{2}$, then under the condition $\sigma_{n-p}(\partial M) >0$, we have
        \begin{equation}\label{lower_bound_Went_second_case}\lambda_{1,p,\beta} (M) \ge \frac{p+1}{p}\sigma_p(\partial M) + \beta \, \min(\sigma_p(\partial M) \sigma_{n-p+1}(\partial M),\sigma_{p+1}(\partial M) \sigma_{n-p}(\partial M), \end{equation}
        with equality when $M$ is a Euclidean ball. Furthermore, if the Ricci curvature of $M$ is nonnegative, then equality occurs in \eqref{lower_bound_Went_second_case} if and only if $M$ is a Euclidean ball. 
            \end{enumerate}
\end{prop}
{\it Proof.} It is based on the inequality \eqref{est_Wen_boundLapl_DtN} and the lower bound of $\nu_{1,p}(M)$ obtained in \cite{RS_DtN2012} and that of $\lambda_{1,p}(\partial M)$ given in \cite{RaulotSavo2011}. Indeed, under our assumption on the curvature operator and if in addition $\min(\sigma_p(\partial M), \sigma_{n-p}(\partial M)) >0$, then using \cite[Thm. 5]{RaulotSavo2011}, we find
\begin{equation}
    \label{low_bound_Hodge_lapl_boundary}
\lambda_{1,p}(\partial M) \ge \min(\sigma_p(\partial M) \sigma_{n-p+1}(\partial M),\sigma_{p+1}(\partial M) \sigma_{n-p}(\partial M). \end{equation}
Assume first $1 \le p < \frac{n+1}{2}$, then $p \le n-p$ which implies that $\sigma_p(\partial M) \le \sigma_{n-p}(\partial M)$. Therefore, \eqref{low_bound_Hodge_lapl_boundary} is satisfied since $\sigma_p(\partial M) >0$.  
Hence, combining \eqref{low_bound_Hodge_lapl_boundary} with the inequality 
$\nu_{1,p} (M) > \frac{n-p+2}{n-p+1}\sigma_p(\partial M)$ \cite[Thm. 1]{RS_DtN2012}, we deduce \eqref{lower_bound_Went_first_case}. When $p \ge \frac{n+1}{2}$, $\sigma_{n-p}(\partial M) \le \sigma_{p}(\partial M)$, so we can apply again \eqref{low_bound_Hodge_lapl_boundary} because $\sigma_{n-p}(\partial M) >0$, which, combined with the estimate $\nu_{1,p} (M) \ge \frac{p+1}{p}\sigma_p(\partial M)$ \cite[Thm. 1]{RS_DtN2012} gives the desired inequality. Moreover, when $M$ is a Euclidean ball, equality holds in \eqref{lower_bound_Went_second_case}, as can be easily seen. Conversely, assume \eqref{lower_bound_Went_second_case} is an equality and  $\textrm{Ric}^M \ge 0$, then \eqref{low_bound_Hodge_lapl_boundary} is an equality as well. Therefore, by the limiting case in \cite[Thm. 5]{RaulotSavo2011}, we deduce that $M$ is a Euclidean ball.  \hfill$\square$ 
\begin{prop}
\label{lower_bound_Went_n}
Assume that the Ricci curvature of $M$ is nonnegative and that the mean curvature $H:=\frac{1}{n} \mathrm{tr}(S)$ of the boundary is bounded below by a nonnegative real number $H_0$. Then 
    \begin{equation}
        \label{est_eigenv_n_went}
    \lambda_{1,n,\beta}(M) \ge (n+1)H_0.\end{equation}
      If $n \ge 2$,  equality is achieved if and only if $M$ is a Euclidean ball.
\end{prop}
{\it Proof.} By using Proposition~\ref{min_Wentz_DtN_Delta_bound} and the estimate $\nu_{1,n}(M) \ge (n+1)H_0$ (\cite[Thm. 2]{RS_DtN2012}), we obtain the stated inequality. If $M$ is a Euclidean ball, equality follows from Theorem~\ref{spectrum_n_Went_ball}. Conversely, equality in \eqref{est_eigenv_n_went} implies equality in \cite[Thm. 2]{RS_DtN2012}, and therefore $M$ is a Euclidean ball.
\hfill$\square$ 

\vspace{0.3em}

Next, we derive upper bounds for the eigenvalues of the Wentzell operator in some specific cases. 
\begin{prop}
\label{up_bound_harm} 
Let $f$ be a harmonic function on $M$ such that $df \ne 0$. Then
     \begin{equation}
\label{upp_bound_lambda_2_0}\lambda_{2,0,\beta}(M) \le \dfrac{\displaystyle\int_{\partial M}  \big(\dfrac{\partial f}{\partial \nu}\big)^2 \, d \mu_g}{\displaystyle \int_M |df|^2 \, d \mu_g} +\beta \, \dfrac{\displaystyle\int_{\partial M}  |d^{\partial M}f|^2 \, d \mu_g \, \displaystyle\int_{\partial M}  \big(\dfrac{\partial f}{\partial \nu}\big)^2 \, d \mu_g}{\displaystyle\left(\int_M |df|^2 \, d \mu_g\right)^2}.\end{equation}
Moreover, equality holds in \eqref{upp_bound_lambda_2_0} if and only if $f_{|_{\partial M}}$ is an eigenfunction of $\mathcal{L}_\beta^{[0]}$ associated with the eigenvalue $\lambda_{2,0,\beta}(M)$, and is also an eigenfunction of both $\mathcal{T}^{[0]}$ and $\Delta^{\partial M}$. 
\end{prop}
{\it Proof.} We adapt an idea from the proof of \cite[Prop. 14]{RS_DtN2012}. Without loss of generality, we may assume that $\int_M f\, d \mu_g=0$, so that $f$ can be taken as a test function for $\lambda_{2,0,\beta}(M)$. Thus
\begin{equation}
\label{char_lambda_2_0}    
\lambda_{2,0,\beta}(M) \, \int_{\partial M} f^2 \, d \mu_g \le
\int_M |df|^2 \, d \mu_g +\beta \, \int_{\partial M}|d^{\partial M} f|^2 \, d \mu_g.\end{equation}
On the other hand, by applying the integration by parts formula and the Cauchy-Schwarz inequality, we obtain
\begin{equation}
\label{harm_C-S}\int_M |df|^2 \, d \mu_g=\int_M f \, \underbrace{\Delta f}_{0} \, d \mu_g - \int_{\partial M} f \, \dfrac{\partial f}{\partial \nu} \, d \mu_g \le \left(\int_{\partial M} f^2 \, d \mu_g\right)^{\frac{1}{2}} \left(\int_{\partial M}  \big(\dfrac{\partial f}{\partial \nu}\big)^2 \, d \mu_g\right)^{\frac{1}{2}}.
\end{equation}
Therefore, $\displaystyle \int_{\partial M} f^2 \, d \mu_g \ge \dfrac{\displaystyle\left(\int_M |df|^2 \, d \mu_g\right)^2}{\displaystyle\int_{\partial M}  \big(\dfrac{\partial f}{\partial \nu}\big)^2 \, d \mu_g}$. We deduce the estimate by substituting the last inequality into \eqref{char_lambda_2_0}.
Equality holds in \eqref{upp_bound_lambda_2_0} if and only if equality is achieved in both \eqref{char_lambda_2_0} and \eqref{harm_C-S}. Equivalently, $f_{|_{\partial M}}$ is simultaneously an eigenfunction of $\mathcal{L}_\beta^{[0]}$ corresponding to the eigenvalue $\lambda_{2,0,\beta}(M)$ and an eigenfunction of $\mathcal{T}^{[0]}$. In this case, $f_{|_{\partial M}}$ is also an eigenfunction of  $\Delta^{\partial M}$. 
\hfill$\square$ 
\begin{remark}
    The extension of Proposition~\ref{up_bound_harm} to differential forms, as in \cite[Prop. 14]{RS_DtN2012}, is more delicate. Although the argument used to prove~\eqref{upp_bound_lambda_2_0} can be adapted, the analogue of \eqref{char_lambda_2_0} contains an additional term involving the boundary codifferential, which is difficult to control.
\end{remark}
As a consequence of Proposition \ref{up_bound_harm}, we obtain the following result. 
\begin{cor}
    Let $f$ be a function on $M$ such that $df$ is a nontrivial parallel form. Then, \eqref{upp_bound_lambda_2_0} holds. Moreover, if the mean curvature $H$ of the boundary is constant, then, equality occurs in \eqref{upp_bound_lambda_2_0} if and only if $f_{|_{\partial M}}$ is an eigenfunction of $\mathcal{L}_\beta^{[0]}$ associated with the eigenvalue $\lambda_{2,0,\beta}(M)$. 
\end{cor}
{\it Proof.} First, the function $f$ is harmonic as $df$ is parallel. Therefore, it remains only to prove the sufficient condition for equality. To this end, suppose that $f_{|_{\partial M}}$ is an eigenfunction of $\mathcal{L}_\beta^{[0]}$ associated with  $\lambda_{2,0,\beta}(M)$. By the fact that $f$ is harmonic and $\nabla^2 f=0$, we obtain $\Delta^{\partial M} f=-nH \dfrac{\partial f}{\partial \nu}$. We then deduce that 
$$-\left(1+\beta \, nH\right)\dfrac{\partial f}{\partial \nu}=\lambda_{2,0,\beta}(M) f,$$
which means that $f$ is an eigenfunction of $T^{[0]}$ and then of $\Delta^{\partial M}$ as well. Therefore, the statement follows from Proposition~\ref{up_bound_harm} \hfill$\square$ 
  \begin{remark}
By  Proposition~\ref{up_bound_harm}, we obtain, under suitable hypotheses, a non-sharp upper bound for the first positive eigenvalue of $\mathcal{L}^{[0]}_\beta$ in terms of the isoperimetric ratio. More precisely, suppose  that there exists a nonconstant harmonic function $f$ on $M$. Then 
    \begin{equation}
        \label{upp_bound_lambda_2_0_relaxed}
           \lambda_{2,0,\beta}(M) < \dfrac{\int_{\partial M} |df|^2 \, d \mu_g}{\int_{M} |df|^2 \, d \mu_g} +\frac{\beta}{4} \, \left(\dfrac{\int_{\partial M} |df|^2 \, d \mu_g}{\int_{M} |df|^2 \, d \mu_g}\right)^2.
   \end{equation}
   In particular, if $df$ is parallel, then
   $$\lambda_{2,0,\beta}(M) < \dfrac{\vol(\partial M)}{\vol(M)} + \frac{\beta}{4}\, \left(\dfrac{\vol(\partial M)}{\vol(M)}\right)^2.$$
To prove \eqref{upp_bound_lambda_2_0_relaxed}, we first observe that along the boundary
$|df|^2
=|d^{\partial M} f|^2 + \big(\dfrac{\partial f}{\partial \nu}\big)^2$.
Consequently, 
\begin{equation}
\label{upp_bound_2_0_1st_term}
    \displaystyle \int_{\partial M}  \big(\dfrac{\partial f}{\partial \nu}\big)^2 \, d \mu_g \le \displaystyle \int_{\partial M} |df|^2 \, d \mu_g.
\end{equation} 
Also, by the inequality $ab \le \frac{1}{2}(a^2+b^2)$ for any real numbers $a$ and $b$ applied to $a=\left(\displaystyle\int_{\partial M} |d^{\partial M} f|^2 \, d \mu_g\right)^{\frac{1}{2}}$ and $b=\left(\displaystyle\int_{\partial M}  \big(\dfrac{\partial f}{\partial \nu}\big)^2 \, d \mu_g\right)^{\frac{1}{2}}$, we deduce that 
\begin{equation}
\label{upp_bound_2_0_2nd_term}\displaystyle\int_{\partial M}  |d^{\partial M}f|^2 \, d \mu_g \, \displaystyle\int_{\partial M}  \big(\dfrac{\partial f}{\partial \nu}\big)^2 \, d \mu_g \le 
\frac{1}{4} \left(\int_{\partial M} |df|^2 \, d \mu_g \right)^2.
\end{equation}
Therefore, we conclude~\eqref{upp_bound_lambda_2_0_relaxed} by combining the inequalities \eqref{upp_bound_2_0_1st_term}, \eqref{upp_bound_2_0_2nd_term} and \eqref{upp_bound_lambda_2_0}. The inequality is strict. Otherwise, the equality cases in \eqref{upp_bound_2_0_1st_term}, 
and \eqref{upp_bound_2_0_2nd_term} would imply $\int_{\partial M}  \big(\dfrac{\partial f}{\partial \nu}\big)^2 \, d \mu_g=0$. In view of \eqref{harm_C-S}, this contradicts the assumption that $df$ is nontrivial. If we assume $df$ parallel, then its pointwise norm is constant. Therefore, the second estimate follows directly from the first one. \\ \\
For other upper bounds involving the isoperimetric ratio, we refer to \cite[Thm. 1.4]{FangXia2021}
and 
\cite[Chap. 4]{Ndiaye2022}. 
\end{remark}
   Next, recall that a compact Riemannian manifold $M$ with smooth nonempty boundary $\partial M$ is called a \emph{harmonic domain} \cite[p. 893]{RS_DtN2012} if it supports a solution $f$ to the Serrin problem
\begin{equation}\label{eq:scalarSerrin}
\left\{\begin{array}{rrl}\Delta f&=1&\textrm{on }M\\ f&=0&\textrm{on }\partial M\\ \displaystyle \frac{\partial f}{\partial \nu}&=c&
\textrm{on }\partial M\end{array}\right.
\end{equation}
for some constant $c\in\R$. It is straightforward that in this case $c=\frac{\mathrm{Vol}(M)}{\mathrm{Vol}(\partial M)}$.
 \begin{prop}
 \label{upp_bound_n_harmonic}
     Let $M$ be a harmonic domain and $f$ be a solution to the Serrin problem \eqref{eq:scalarSerrin}. 
     The real number $\lambda = \frac{\mathrm{Vol}(\partial M)}{\mathrm{Vol}(M)}$ is an eigenvalue of $\mathcal{L}_\beta^{[n]}$ 
     In particular, 
     $$\lambda_{1,n,\beta}(M) \le \frac{\mathrm{Vol}(\partial M)}{\mathrm{Vol}(M)}.$$
 \end{prop}
{\it Proof.}  We introduce the function $\phi := \displaystyle \frac{\partial f}{\partial \nu}=c$ on $\partial M$. It is clear that the $1$-form $df$ is the harmonic normal extension of $\phi$. Therefore, 
$$\mathcal{L}_{\beta,D}^{[0]} (\phi)=\iota^* (\delta df) +\beta \, \Delta^{\partial M} \phi =1.$$
It follows that $\phi$ is an eigenfunction of $\mathcal{L}_{\beta,D}^{[0]}$ associated with the eigenvalue $\frac{1}{c}$, which implies that $\frac{\mathrm{Vol}(\partial M)}{\mathrm{Vol}(M)}$ is an eigenvalue of $\mathcal{L}_{\beta}^{[n]}$. The proof is then complete.
\hfill$\square$

\vspace{0.3em}

Let 
$$\|S\|^2_p := \max_{i_1 < \dots < i_p} (\eta_{i_1}^2 + \dots + \eta_{i_P}^2),$$ where $\eta_1, \dots , \eta_n$ are the principal curvatures. We recall the following pointwise estimates from \cite{RaulotSavo2011}, valid for every $\phi \in \Omega^p(\partial M)$:
\begin{equation}
    \label{min_Sp}
    \<S^{[p]} (\phi),\phi\> \ge \sigma_p |\phi|^2,
\end{equation}
\begin{equation}
    \label{maj_Sp}
    |S^{[p]} (\phi)|^2 \le  p \, \|S\|_p^2 \, |\phi|^2.
\end{equation}
Recall also the following identities \cite[Lem.~18]{RaulotSavo2011}
\begin{equation}
   \label{eq:d_delta_RS} 
\left\{
\begin{array}{ll}
 \delta^{{\partial M}}(\iota^* \omega)= \iota^* (\delta\omega)+i_\nu\nabla_{\nu}\omega+ S^{[p-1]}(\nu\lrcorner\, \omega)-nH\nu\lrcorner\,\omega \\
d^{{\partial M}}(\nu\lrcorner\,\omega)= -\nu\lrcorner\, d\omega + \iota^* (\nabla_{\nu}\omega) - S^{[p]}(\iota^* \omega).
\end{array}
\right.
\end{equation}
\begin{prop}
\label{upp_bound_arb_degree_harmonic}
    Let $M$ be a harmonic domain. Assume that $M$ admits a parallel $p$-form $\omega_0$. Then
    $$\min(\lambda_{1,p-1,\beta}(M),\lambda_{1,n-p,\beta}(M)) \le \frac{\mathrm{Vol}(\partial M)}{\mathrm{Vol}(M)} - \min(\sigma_p ,\sigma_{n+1-p} )  +\beta \, \max\left( (n-p+1)\|S\|^2_{n+1-p}, p \|S\|^2_p  \right).$$
\end{prop}
{\it Proof.} Let $f$ be the solution to \eqref{eq:scalarSerrin}, and set $\omega:=f\, \omega_0$. By \cite[Prop.~2.1]{EGHM}, $\omega$ satisfies $\Delta \omega= \omega_0$ on $M$, and,  $\omega_{|_{\partial M}}=0$, $\nu \lrcorner\, d \omega=c \iota^* \omega_0$ and $\iota^* (\delta \omega)=-c \, \nu \lrcorner\, \omega_0$ on $\partial M$.  
Since $\iota^* d \omega=0$, we can use $d \omega$ as a test form in the variational characterization~\eqref{var_char_dual} of $\lambda_{1,p,\beta}^D(M)$. Therefore
\begin{equation}
\label{lambda_1p_D-dom_harm}
  \|\nu \lrcorner\, d \omega\|^2_{L^2(\partial M)}\,   \lambda_{1,p,\beta}^D(M)\le  \|\delta d \omega\|^2_{L^2(M)}+\beta \,\left(\|d^{\partial M} (\nu \lrcorner\, d \omega)\|^2_{L^2(\partial M)}+\|\delta^{\partial M} (\nu \lrcorner\, d \omega)\|^2_{L^2(\partial M)}\right). \end{equation}
    By the partial integration formula
    \begin{align*}
        \|\delta d \omega\|^2_{L^2(M)}&=\int_M \< \delta d \omega, \delta d \omega \> \, d \mu_g \\
        &= \int_M \<\underbrace{d \Delta\omega}_{0} , d \omega \> \, d \mu_g + \int_{\partial M} \< \iota^* \delta d \omega , \nu \lrcorner \, d \omega \>\, d \mu_g.
    \end{align*}
Writing $\iota^*(\delta d \omega) = \iota^*(\Delta \omega)- \iota^*(d \delta  \omega)$, and using the properties of $\omega$ recalled above, we obtain $\iota^*(\delta d \omega) = \iota^* \omega_0 + c \, d^{\partial M} (\nu \lrcorner \, \omega_0)$, and \eqref{eq:d_delta_RS} then yields $\iota^*(\delta d \omega) = \iota^* \omega_0 -c \, S^{[p]}(\iota^* \omega_0)$. 
It follows that 
$$ \|\delta d \omega\|^2_{L^2(M)}= c  \|\iota^* \omega_0\|^2_{L^2(\partial M)}-c^2 ( S^{[p]}(\iota^* \omega_0), \iota^* \omega_0 )_{L^2(\partial M)}.$$
Next, $ d^{\partial M} (\nu \lrcorner\, d \omega)=c \iota^* (d \omega_0)=0$. Finally, we consider the term $\|\delta^{\partial M} (\nu \lrcorner\, d \omega)\|^2_{L^2(\partial M)}$. 
\begin{align*}
    \|\delta^{\partial M} (\nu \lrcorner\, d \omega)\|^2_{L^2(\partial M)} =& c^2 \|\delta^{\partial M} (\iota^* \omega_0)\|^2_{L^2(\partial M)} \\
    \stackrel{\eqref{eq:d_delta_RS}}{=}& c^2 \, \|S^{[p-1]}(\nu \lrcorner\, \omega_0) -nH \, \nu \lrcorner \, \omega_0\|^2_{L^2(\partial M)}\\
    =&c^2 \,\|S^{[n+1-p]}(\iota^* \star \omega_0) \|^2_{L^2(\partial M)}.
\end{align*}
The last equality follows from the identity
$\star_{\partial M}S^{[p]}+S^{[n-p]}\star_{\partial M}=nH\, \star _{\partial M}$. In view of the previous computations, inequality~\eqref{lambda_1p_D-dom_harm} becomes
\begin{align*}
     \|
    \iota^*\omega_0\|^2_{L^2(\partial M)}  \lambda_{1,p,\beta}^D(M)\le&  \dfrac{1}{c}\|
    \iota^*\omega_0\|^2_{L^2(\partial M)} - ( S^{[p]}(\iota^* \omega_0), \iota^* \omega_0 )_{L^2(\partial M)} +\beta   \,\|S^{[n+1-p]}(\iota^* \star \omega_0) \|^2_{L^2(\partial M)}\\
     \le & \dfrac{1}{c}\|
    \iota^*\omega_0\|^2_{L^2(\partial M)} -\sigma_p \|
    \iota^*\omega_0\|^2_{L^2(\partial M)} + \beta \, (n+1-p) \|S\|^2_{n+1-p} \, \|
    \nu \lrcorner\, \omega_0\|^2_{L^2(\partial M)}.
    . \end{align*}
     Repeating the above computation with $\star \omega_0$ in place of $\omega_0$, and adding the resulting inequality to the preceding one, we obtain
     \begin{align*}
\min(\lambda_{1,p,\beta}^D(M),\lambda_{1,n+1-p,\beta}^D(M)) \le &  
    \frac{1}{c} - \min(\sigma_p ,\sigma_{n+1-p} )  +\beta \, \max\left( (n+1-p)\|S\|^2_{n+1-p}, p \|S\|^2_p  \right). 
\end{align*}
The desired estimate follows from the preceding inequality and the identity
 $\lambda_{1,q,\beta}^D(M)=\lambda_{1,n-q,\beta}(M)$ for each $q=0, \dots,n$. \hfill$\square$
\begin{remark}
In the particular case where $\beta=0$, we obtain the following estimate for the Dirichlet-to-Neumann eigenvalues when $M$ is a harmonic domain carrying a parallel $p$-form:
$$\min(\nu_{1,p-1}(M),\nu_{1,n-p}(M)) \le \frac{\mathrm{Vol}(\partial M)}{\mathrm{Vol}(M)} - \min(\sigma_p ,\sigma_{n+1-p} ) .$$
This upper bound contains an additional term compared to that in \cite[Thm. 5.2]{EGHM}, which is  sharper when both $\sigma_p$ and $\sigma_{n+1-p}$ are positive.  
\end{remark}

\section{Gap estimate}
\label{Sec:Wentz_gap_estimate}

In this section, we assume the manifold $M^{n+1}$ isometrically immersed 
into $\R^d$ for somr $d \ge n+1$. We will estimate the gap between the first eigenvalues of the Wentzell operator for consecutive degrees. Our approach follows \cite[Thm. 2.3]{guerinisavo2004} (see also \cite[Thm. 5.8]{elchamiginouxhabib21}, \cite[Thm. 4]{RS_DtN2012} and \cite{Michel2019}).
\\\\
We denote by $\iota_0 : M \rightarrow \R^d$ the immersion of $M$ into $\R^d$, $II$ the second fundamental form of the 
immersion. For any normal vector field $N$ to $M$ in $\R^d$, we let $II_N$ be the associated Weingarten map defined by 
\[\< II_N(X),Y\>=\<  N,II(X,Y)\>\]
for all $X,Y$ tangent to $M$. The endomorphism $II_N$ admits a  canonical extension to differential forms, in the same way as the extension of $S$ given in \eqref{ext_S}.
The boundary $\partial M$ is also isometrically immersed into $\R^d$, the immersion being given by $\Bar{\iota}=\iota_0 \circ \iota$. The second fundamental form $\Bar{II}$ of $\Bar{\iota}$ is given by:
$$\Bar{II} (X,Y)= II (X,Y) + S(X,Y),$$
for all vector fields $X,Y$ on $\partial M$, where $S$ is the second fundamental form of $\iota$. 
Consider a local orthonormal frame $\{e_1, \dots , e_n , e_{n+1}=\nu\}$ of $TM$ and a local orthonormal frame $\{\nu_1, \dots , \nu_m\}$ of the normal bundle $T^\perp M$with $m=d-n-1$. The operator $T^{[p]}$ defined on $\Omega^p(M)$ by
\[T^{[p]}:=\sum_{a=1}^m (II^{[p]}_{\nu_a})^2,\]
is independent of the chosen o.n.b. of $T^\perp M$, and is also self-adjoint and nonnegative. We also consider the operator $\Bar{T}^{[p]}$ defined on $\Omega^p(\partial M)$ by
\[\Bar{T}^{[p]}:= (S^{[p]})^2 +\sum_{a=1}^m (\Bar{II}^{[p]}_{\nu_a})^2\]
 Let $x_1, \dots , x_{d}$ be the standard coordinates on $\R^{d}$. For $i=1, \dots, d$, the unit parallel vector field $\partial_{x_i}$ splits into $\partial_{x_i}=\partial_{x_i}^T +\partial_{x_i}^\perp\in T  M \oplus T^\perp  M$. Also, for each $i$, the vector field $\partial_{x_i}^T$ can be written as $\partial_{x_i}^T=\Bar{\partial}_{x_i} + g(\partial_{x_i}^T, \nu) \nu$, where $\Bar{\partial}_{x_i}$ is tangent to the boundary $\partial M$. 
\\\\
The following lemma can be found in \cite{guerinisavo2004,Michel2019}. See also the proof of Thm. 5.8 in \cite{elchamiginouxhabib21}.
\begin{lemma} 
\label{lem:terms_imm_isom}
    For every $p$-form $\omega$ on $M$, we have:
    \begin{enumerate}
        \item $\sum_{i=1}^d |\partial_{x_i}^T\lrcorner \, \omega|^2=p|\omega|^2$.
        \item $\sum_{i=1}^d |d(\partial_{x_i}^T\lrcorner \, \omega)|^2=|\nabla \omega|^2 +\<  T^{[p]} \omega , \omega \> + (p-1)|d \omega|^2$.
        \item $\sum_{i=1}^d |\delta(\partial_{x_i}^T\lrcorner \,  \omega)|^2= (p-1) |\delta \omega|^2$.
    \end{enumerate}
\end{lemma}
Lemma \ref{lem:terms_imm_isom} yields the following result.
\begin{lemma} 
\label{lem:terms_imm_isom_bound}
Let $\omega$ be a $p$-form on $M$ such that $\nu  \lrcorner \, \omega=0$. Then, along the boundary, we have:
\begin{enumerate}
    \item $\sum_{i=1}^d |\iota^*(\partial_{x_i}^T\lrcorner \, \omega)|^2=p|\iota^* \omega|^2$.
    \item $\sum_{i=1}^d |d^{\partial M} \iota^*(\partial_{x_i}^T\lrcorner \, \omega)|^2=|\nabla^{\partial M} \iota^* \omega|^2 + (p-1)|d^{\partial M} (\iota^*\omega)|^2 +\<  \Bar{T}^{[p]} (\iota^* \omega) , \iota^* \omega \>$.
        \item $\sum_{i=1}^d |\delta^{\partial M} \iota^*(\partial_{x_i}^T\lrcorner \,  \omega)|^2= (p-1) |\delta^{\partial M} \iota^* \omega|^2$.
\end{enumerate}
    \end{lemma}
\noindent {\it Proof.} 
    For each $i=1, \dots , d$, $\iota^*(\partial_{x_i}^T\lrcorner \, \omega)=\Bar{\partial}_{x_i} \lrcorner \, \iota^* \omega$. Indeed, for all vector fields $X_1, \dots , X_{p-1}$ tangent to the boundary,
    \begin{eqnarray*}
        \iota^*(\partial_{x_i}^T\lrcorner \, \omega) (X_1, \dots , X_{p-1}) 
        &=& \omega (\partial_{x_i}^T, X_1, \dots , X_{p-1})\\
        &=& g(\partial_{x_i}^T, \nu) \underbrace{\omega (\nu, X_1, \dots , X_{p-1})}_{0} + \iota^* \omega (\Bar{\partial}_{x_i}, X_1, \dots , X_{p-1}) \\
        &=& (\Bar{\partial}_{x_i} \lrcorner \, \iota^* \omega) ( X_1, \dots , X_{p-1}).
    \end{eqnarray*}
 The desired result then follows directly from Lemma \ref{lem:terms_imm_isom}.
    \hfill$\square$
    \begin{thm}\label{thm:estimate_isom_imm}
Let $(M^{n+1},g)$ be a compact 
Riemannian manifold with smooth boundary $\partial M$. Assume that $M^{n+1}$ is isometrically immersed into $\R^{d}$. Then for every $p$-form $\omega$ on $M$ satisfying $\nu \lrcorner \, \omega=0$, we have
\begin{eqnarray}
    \label{eq:estimate_isom_imm}
    p \lambda_{1,p-1, \beta}(M)\int_{\partial M} |\omega|^2 \, d \mu_g &\le& \int_M \left( |\nabla \omega|^2 + \<  T^{[p]} \omega, \omega \> + (p-1) (|d \omega|^2 + |\delta \omega|^2)\right) \, d \mu_g \nonumber \\
    && + \beta \, \int_{\partial M} \left( |\nabla^{\partial M} \iota ^* \omega|^2 + \<  \Bar{T}^{[p]} \iota ^* \omega, \iota ^* \omega \> \right) \, d \mu_g \nonumber \\
    && + \beta \,(p-1) \int_{\partial M} \left(   |d^{\partial M} \iota ^* \omega|^2 + |\delta^{\partial M} \iota ^*\omega|^2\right) \, d \mu_g.
\end{eqnarray}
\end{thm}
\noindent {\it Proof.} 
    For each $i=1, \dots , d$, the $(p-1)$-form $\partial_{x_i}^T\lrcorner \, \omega$ satisfies $\nu\lrcorner\, (\partial_{x_i}^T\lrcorner \, \omega)=0$. Then, by the variational characterization \eqref{variat_char}, we obtain
\begin{eqnarray*}
\lambda_{1,p-1, \beta}(M)\int_{\partial M} |\partial_{x_i}^T\lrcorner \, \omega|^2 \, d \mu_g &\le& \int_M \left(|d (\partial_{x_i}^T\lrcorner \, \omega)|^2+|\delta (\partial_{x_i}^T\lrcorner \, \omega)|^2\right)\, d\mu_g\\
&&+\beta \, \int_{\partial M}\left(|d^{\partial M} \iota^* (\partial_{x_i}^T\lrcorner \, \omega)|^2+|\delta^{\partial M} \iota^* (\partial_{x_i}^T\lrcorner \, \omega)|^2\right) \, d \mu_g
\end{eqnarray*}
Summing over $i=1, \dots, d$ and applying Lemmas \ref{lem:terms_imm_isom} and \ref{lem:terms_imm_isom_bound}, we obtain the claimed inequality.
    \hfill$\square$
\begin{thm}\label{gap_estimate}
Let $(M^{n+1},g)$ be a compact 
Riemannian manifold with boundary, isometrically immersed in $\R^{d}$. Assume that $W^{[p]}-T^{[p]} \ge 0$ for some $p=1, \dots , n$. 
Then 
\begin{equation}
\label{est:gap_Wentzell}\lambda_{1,p,\beta}(M)-\lambda_{1,p-1, \beta}(M)\geq \frac{1}{p}\left(\sigma_p + \beta \,\inf_{\partial M}(\Bar{W}^{[p]}-\Bar{T}^{[p]}) \right),\end{equation}
where $\bar{W}^{[p]}$ is the Bochner operator of $\partial M$. 
\end{thm}
{\it Proof.} 
Let $\phi$ be an eigenform of the Wentzell operator associated with the eigenvalue $\lambda_{1,p,\beta}(M)$. Without loss of generality, we may assume that $\|\phi\|_{L^2(\partial M)}=1$. We denote by $\hat{\phi}$ the harmonic tangential extension of $\phi$ to $M$. The Reilly formula \cite[Thm. 3]{RaulotSavo2011} applied first to $\hat{\phi}$ on $M$ and then to $\phi$ on $\partial M$, leads to
$$\int_M |\nabla \hat{\phi}|^2 \, d \mu_g = \int_M (|d \hat{\phi}|^2 + |\delta \hat{\phi}|^2 -\<  W^{[p]} \hat{\phi}, \hat{\phi} \> ) \, d \mu_g - \int_{\partial M} \<  S^{[p]} {\phi}, {\phi} \>  \, d \mu_g,$$
and
$$\int_{\partial M} |\nabla^{\partial M} {\phi}|^2 \, d \mu_g = \int_{\partial M} (|d^{\partial M} {\phi}|^2 + |\delta^{\partial M} {\phi}|^2 -\<  \Bar{W}^{[p]} {\phi}, {\phi} \> ) \, d \mu_g .$$
Applying \eqref{eq:estimate_isom_imm} to $\hat{\phi}$ and then substituting the previous identities, we obtain
\begin{eqnarray*}
p \lambda_{1,p-1, \beta}(M)  \le  p \lambda_{1,p, \beta} (M)  + \int_M \<  (T^{[p]} - W^{[p]}) \hat{\phi}, \hat{\phi} \>  \, d \mu_g   + \int_{\partial M} (\beta \, \<  (\Bar{T}^{[p]} -\Bar{W}^{[p]}) {\phi}, {\phi} \> - \<  S^{[p]} {\phi}, {\phi} \> )  \, d \mu_g .   
\end{eqnarray*}
Therefore, using the assumption $W^{[p]}-T^{[p]} \ge 0$ and the inequality 
$\langle S^{[p]}(\phi), \phi \rangle \ge \sigma_p  |\phi|^2$, 
we obtain the estimate. 
\hfill$\square$
\begin{cor}
   Let $M$ be any compact domain in $\R^{n+1}$. Then, the estimate~\eqref{est:gap_Wentzell} holds for all $p=1, \dots , n$. Moreover, when $M$ is the unit Euclidean ball of $\R^{n+1}$, equality occurs for $p \ge \frac{n+3}{2}$.
\end{cor}
{\it Proof.} The first statment is immediate. Assume now $M=\mathbb{B}^{n+1}$. In this case, $W^{[p]}=T^{[p]}=0$, $\sigma_p=p$, $\Bar{W}^{[p]}=p(n-p)$ and $\Bar{T}^{[p]} =p^2$. The r.h.s. of \eqref{est:gap_Wentzell} is equal to $1+\beta \, (n-2p)$. Let us now compute the l.h.s. of \eqref{est:gap_Wentzell}, i.e., the gap between the first eigenvalues of the Wentzell operator in degrees $p$ and $p-1$. By applying Corollary \ref{cor:first_eigen_ball}, we obtain:
$$
\lambda_{1,p,\beta}(\mathbb{B}^{n+1})- \lambda_{1,p-1, \beta}(\mathbb{B}^{n+1})=\begin{cases}
\medskip \frac{n+3}{n+1}+\beta \, (n+2-2p) & \textrm{for }    p < \frac{n+1}{2}, \\
 \medskip \frac{n+3}{n+1} & \textrm{for }    p = \frac{n+1}{2}, \\
 \medskip 1+\beta (n-2p) & \textrm{for }    p \ge \frac{n+3}{2}.
\end{cases}
$$
It is clear that equality is achieved only in the last case.
\hfill$\square$

\vspace{0.3em}

Another consequence of Theorem~\ref{gap_estimate} is the following:
\begin{cor}
   Let $M$ be any compact domain in $\R^{n+1}$. If $\sigma_p\geq 
0$ and $\Bar{W}^{[p]}-\Bar{T}^{[p]} \ge 0$ for all $p=1, \dots , n$, then, 
   $$\lambda_{1,1,\beta}(M) \le \lambda_{1,2,\beta}(M) \le \dots \le \lambda_{1,n,\beta}(M).$$
\end{cor}


\begin{thebibliography}{99}

\bibitem{BL1999} 
A.~Bendali and K.~Lemrabet, \emph{The effect of a thin coating on the scattering of a time-harmonic wave for the Helmholtz equation}, SIAM J. Appl. 
Math. {\bf 56} (6) (1996) 1664–1693.


\bibitem{DKL2016} 
M.~Dambrine, D.~Kateb and J.~Lamboley, \emph{An extremal eigenvalue problem for the Wentzell-Laplace operator}
Ann.~I.~H.~Poincaré, Analyse non linéaire, \textbf{33} (2016) no. 2, 409--450.

\bibitem{DengDuMaoZhao2026}
Y.~Deng, F.~Du, J.~Mao, Y.~Zhao, \emph{Sharp eigenvalue estimates and related rigidity theorems},   Hokkaido Math. J. \textbf{55} (1) (2026), 57--85.

\bibitem{DuWangXia2018}
F.~Du, Q.~Wang and C.~Xia, \emph{
Estimates for eigenvalues of the Wentzell–Laplace operator},
Journal of Geometry and Physics,
\textbf{129} (2018),  25--33.

\bibitem{DuMaoWangXia2018}
F.~Du, J.~Mao, Q.~Wang and C.~Xia, \emph{
Isoperimetric bounds for eigenvalues of the
Wentzell-Laplace, the Laplacian and a biharmonic
Steklov problem}, \href{https://arxiv.org/abs/1808.10578v2}
{arXiv:1808.10578v2}.







\bibitem{elchamiginouxhabib21}
F.~El Chami, N.~Ginoux and G.~Habib, \emph{New eigenvalue estimates involving Bessel functions}, Publ. Mat.
\textbf{65} (2021), 681--726

\bibitem{EGHM}
F.~El Chami, N.~Ginoux, G.~Habib and O.~Makhoul, \emph{Biharmonic Steklov operator on differential forms}, Annales math\'ematiques Blaise Pascal \textbf{31} (2024) no. 2, pp. 189-237. 


\bibitem{FangXia2021}
F.~Fang and C.~Xia, \emph{
Isoperimetric Bounds for Lower Order Eigenvalues}, \href{https://arxiv.org/abs/2110.01869}
{arXiv:2110.01869}.

\bibitem{Feller1952} 
W.~Feller, \emph{The parabolic differential equations and the associated semi-groups of transformations}, Ann. of Math. \textbf{55} (3) (1952), 468--519.

\bibitem{Feller1954}
W.~Feller, \emph{Diffusion Processes in One Dimension}, Trans. Amer. Math. Soc. {\bf 77} (1954), no. 1, 1--31. 

\bibitem{GM}
S.~Gallot and D.~Meyer, \emph{Op\'erateur de courbure et laplacien des formes diff\'erentielles d'une vari\'et\'e riemannienne}, J. Math. Pures Appl.  \textbf{54} (1975), 259--284.


\bibitem{guerinisavo2004}
P.~Gu\'erini and A.~Savo, {\it Eigenvalue and gap estimates for the Laplacian acting on $p$-forms}, Trans. Amer. Math. Soc. {\bf 356} (2004), no. 1, 319--344.




\bibitem{Michel2019}
D.~Michel, \emph{Eigenvalue and gap estimates of isometric immersions for the 
Dirichlet-to-Neumann operator acting on $p$-forms}, C.~R.~Acad. Sci. Paris, Ser. I  \textbf{357} (2019), no. 2, 180--187.



\bibitem{Ndiaye2022}
A.M.~Ndiaye, \emph{Extrinsic Estimates of Eigenvalues for the Laplacian with Wentzel
Boundary Condition}, PhD thesis, Universit\'e de Neuch\^atel, 2022. 


\bibitem{RaulotSavo2011}
S.~Raulot and A.~Savo, \emph{A Reilly formula and eigenvalue estimates for
differential forms}, J. Geom. Anal. \textbf{21} (2011), no. 3, 620--640.

\bibitem{RS_DtN2012} 
S.~Raulot and A.~Savo, \emph{On the first eigenvalue of the Dirichlet-to-Neumann operator on forms}, J. Funct. Anal., \textbf{262} (3) (2012), 889--914.

\bibitem{RS_DtN2014} 
S.~Raulot and A.~Savo, \emph{On the spectrum of the Dirichlet-to-Neumann operator acting on forms of a Euclidean domain},
J. Geom. Phys., \textbf{77} (2014), 1--12.




\bibitem{Sc}
G.~Schwarz, \emph{Hodge decomposition - A method for solving
boundary value problems}, Lecture notes in Mathematics, Springer, 1995.


\bibitem{TaylorPDEII}
M.E.~Taylor,  \emph{Partial Differential Equations II}, Qualitative Studies of Linear Equations, Springer, 2023.

\bibitem{Wentzell1959}
A.D.~Venttsel, \emph{On Boundary Conditions For Multidimensional Diffusion Processes}, Theory of Probability and its Applications, \textbf{4} (2) (1959), 164--177.

\bibitem{XiaWang2018}
C.~Xia and Q.~Wang, \emph{Eigenvalues of the Wentzell–Laplace operator and of the fourth order Steklov problems}, J.~Differential Equations, \textbf{264} (2018), 6486--6506.





\end{thebibliography}
\end{document}